\documentclass[usletter,11pt]{amsart}

\usepackage[english]{babel}
\usepackage{amsmath,amssymb}
\usepackage[all]{xy}
\usepackage{graphicx}
\usepackage{xcolor}
\usepackage{listings}
\usepackage[]{url}
\usepackage{todonotes}
\usepackage{kantlipsum} 

\newtheorem{theorem}{Theorem}[section]
\newtheorem{corollary}[theorem]{Corollary}

\newtheorem{lemma}[theorem]{Lemma}
\newtheorem{proposition}[theorem]{Proposition}

\newtheorem{example}[theorem]{Example}
\newtheorem{remark}[theorem]{Remark}

\newtheorem{definition}[theorem]{Definition}

\calclayout

\newcommand{\N}{{\mathbb{N}}}

\newcommand{\Z}{{\mathbb{Z}}}

\newcommand{\Ap}{\operatorname{Ap}}
\newcommand{\wAp}{\widehat{\operatorname{Ap}}}

\newcommand{\MED}{\operatorname{MED}}

\newcommand{\PF}{\operatorname{PF}}

\title[On the computation of the MED closure]{On the computation of  the MED closure of a numerical semigroup}

\author[J. Jiménez--Urroz]{Jorge Jiménez--Urroz}
\address{Departamento de Matem\'aticas e Informática Aplicadas a las Ingenierías Civil y Naval, E.T.S. de Ingeniería de Caminos, Canales y Puertos, Universidad Politécnica de Madrid. Calle Prof. Aranguren 3, 28040 Madrid, Spain.}

\author[J.M. Tornero]{Jos\'e M. Tornero}
\address{Departamento de \'Algebra, Facultad de Matem\'aticas and IMUS, Universidad de Sevilla. Avda. Reina Mercedes s/n, 41012 Sevilla, Spain.}
\thanks{This work was supported by the \emph{Ministerio de Ciencia e Innovaci\'on} under Project PID2020-114613GB-I00 (MCIN/AEI/10.13039/501100011033).}

\date{\today}

\begin{document}

\begin{abstract}
Maximally embedding dimension (MED) numerical semigroups are a wide and interesting family, with some remarkable algebraic and combinatorial properties. Associated to any numerical semigroup one can construct a MED closure, as it is well--known. This paper shows two different explicit methods to construct this closure which also sheds new light on the very nature of this object.
\end{abstract}

\maketitle

\section{Essential concepts}

A semigroup is a pair $(X,\star)$, where $X$ is a set and $\star$ is an associative internal operation. Actually we will be considering monoids, that is, semigroups with unit element, but there are no substantial differences for our concerns. We will be particularly interested in the so--called numerical semigroups. Useful references for the basic concepts and for the results not proved here are \cite{GR,RA}.

\begin{definition}
A numerical semigroup is an additive monoid $S \subset \Z_{\geq 0}$ such that $\Z_{\geq 0} \setminus S$ is a finite set.
\end{definition}

\begin{example}
The first natural example of a semigroup in $\Z_{\geq 0}$ is the semigroup generated by a set $\{ a_1, \ldots ,a_k \big\} \subset \Z_{\geq 0}$, which is the set of linear combinations of these integers with non--negative integral coefficients:
$$
\langle a_1, \ldots ,a_k \rangle = \left\{ \lambda_1 a_1 + \cdots + \lambda_k a_k \; | \; \lambda_i \in \Z_{\geq 0} \right\}.
$$

It turns out that this example is in fact the general case for a numerical semigroup.
\end{example}

\begin{proposition}
Every numerical semigroup $S$ can be written in the form 
$$
S = \langle a_1, \ldots ,a_k \rangle,
$$ 
with $\gcd(a_1, \ldots ,a_k)=1$. We will call $\{a_1, \ldots ,a_k\}$ a set of generators of $S$. There exists a unique minimal set of generators of $S$.
\end{proposition}

Note that, as $\Z_{\geq 0} \setminus S$ is a finite set, there exists $N \in \Z$ such that $x \in S$, for all $x \geq N$.

As $S = \langle \, a_1, \ldots ,a_k \, \rangle$ is nothing but the set of non--negative integers that can be written as a linear combination (with non--negative coefficients) of $\{ a_1, \ldots ,a_k \big\}$, the elements of $S$ are often called \textit{representable integers} (w.r.t. $\{ a_1, \ldots ,a_k \big\}$). In the same fashion the elements of the (finite) set $\Z_{\geq 0} \setminus S$ are called \textit{non--representable integers} or \textit{gaps}.

\begin{definition}
Some important invariants associated to a numerical semigroup $S$ are:

\begin{itemize}
\item The set of gaps, which is the finite set $\Z_{\geq 0} \setminus S$, noted $\operatorname{G}(S)$. Its cardinal is called the genus of $S$, noted $g(S)$.
\item The Frobenius number of $S$ which is the maximum of $\operatorname{G}(S)$, noted $f(S)$. The conductor of $S$ is $c(S) = f(S)+1$, and it is the smallest number verifying $c(S) + \Z_{\geq0} \subset S$.
\item The set of sporadic (or left) elements, noted $N(S)$, which are elements of $S$ smaller than $f(S)$, that is $N(S) = S \cap [0,f(S)]$. Its cardinal is noted $n(S)$ (this invariant has not a properly stablished name in the literature) and verifies directly $n(S)+g(S) = c(S)$.
\item The set of pseudo-Frobenius numbers, which is
$$
\PF(S) = \big\{ x \in \operatorname{G}(S) \; | \; x+s \in S, \; \forall s \in S \setminus \{ 0 \} \big\}.
$$
Its cardinal is called the type of $S$, noted $t(S)$. Clearly $f(S) \in \PF(S)$.
\item The multiplicity of $S$, noted $m(S)$, which is the smallest non--zero element in $S$ (obviously a generator in any case).
\item The embedding dimension of $S$, noted $e(S)$, which is the cardinal of the minimal set of generators.
\end{itemize}
\end{definition}

\ 

An important object related to an element of a numerical semigroup is the so-called Apéry set, named after Roger Apéry and his seminal work relating numerical semigroups and resolution of curve singularities \cite{Apery}.

\begin{definition}
With the previous notations, the Apéry set of $S$ with respect to an element $s \in S \setminus \{0\}$ can be defined as
$$
\Ap(S,s) = \big\{ 0 = w_0, w_1 , \ldots , w_{s-1} \big\}
$$
where $w_i$ is the smallest element in $S$ congruent with $i$ modulo $s$. 
\end{definition}

In particular, mind that $\# \Ap(S,s) = s$. This set has an interesting equivalent definition, which we will review now.

\begin{lemma} 
With the previous notations, for all $s \in S \setminus \{0\}$,
$$
\Ap(S, s) = \big\{x  \in S \; | \; x-s \notin S \big\}
$$
\end{lemma}

\begin{example} 
Let $S_1 = \langle 7, 9, 11, 15 \rangle$. Some of the Apéry sets associated to its generators are:
\begin{eqnarray*}
\Ap(S_1, 7) &=& \big\{ 0, 9, 11, 15, 20, 24, 26 \big\} \\ 
\Ap(S_1, 15) &=& \big\{ 0, 7, 9, 11, 14, 16, 18, 20, 21, 23, 25, 27, 28, 32, 34 \big\}
\end{eqnarray*}
\end{example}

Note that, from the definition, the minimal set of generators of $S$ must be included in every Apéry set, except when the element is a generator itself, in which case it is replaced by $0$ in the Apéry set. Because of this dichotomy we will sometimes use the following variant
$$
\wAp(S, s) = \big\{x  \in S \; | \; x-s \in \operatorname{G}(S) \cup \big\{ 0 \big\} \big\},
$$
which only changes $0$ by $s$ itself, and verifies that the minimal set of generators of $S$ must be included in every $\wAp(S, s)$. This implies, in particular, that $e(S) \leq m(S)$, as $m(S) = \# \Ap(S,m(S))  = \# \wAp(S,m(S))$.

\ 

If we define the partial ordering in $\Z$ given by $a \leq_S b$ if and only if $b-a \in S$, then the Apéry set verifies the following \cite{FGH}:
\begin{itemize}
    \item The Apéry sets $\Ap(S,s)$ are isolated w.r.t $\leq_S$ that is, if $x \in \Ap(S,s)$ and $y \leq_S x$, then $y \in \Ap(S,s)$.
    \item The minimal set of generators is, for all $s \in S$, $\operatorname{Min}_{\leq_S} \Ap(S,s)$.
    \item The set of pseudo--Frobenius elements is, for all $s \in S$, 
    $$
    \operatorname{PF} (S) = \left\{ x - s \; \Big| \; x \in \operatorname{Max}_{\leq_S} \Ap(S,s) \right\}.
    $$
\end{itemize}

The Apéry set can be explicitely computed \cite{GAP,MOT}, although its computation can be shown to be NP--hard, as it implies the calculation of $f(S)$ \cite{RA2}.

\ 

Finally, let us recall the following definition.

\begin{definition}
    Let $S$ be a numerical semigroup. A set $I \subset S$ is called an ideal (of $S$) if $I+S \subset I$, that is, for all $x \in I$ and $y \in S$, we have $x+y \in I$.

    Coherently, the conductor of the ideal, noted $c(I)$, is the smallest integer such that $c(I) + \Z_{\geq0} \subset I$.
\end{definition}

\section{MED closure of a numerical semigroup}

This paper will deal with a specific family of numerical semigroups, which we proceed to define. The results not proved here can be found in \cite{GR}.

\begin{definition}
Let $S$ be a numerical semigroup. We will say that $S$ is a maximal embedding dimension (MED) semigroup if $e(S) = m(S)$. Equivalently, $S$ is a MED semigroup when $\wAp(S,m(S))$ is the minimal generating set of $S$.
\end{definition}

MED semigroups are, not only a generalization of the two important families (Arf and saturated numerical semigroups), but also an interesting class of semigroups to work with on its own right. 

Among some remarkable properties of these semigroups we would mention that if $S$ is a MED semigroup with minimal generating set $a_1 < \ldots < a_k$, then we have $f(S) = a_k - a_1$. As noted above, the problem of computing $f(S)$ from $a_1, \ldots, a_k$ is NP--hard for general semigroups \cite{RA2}. 

In the same vein, the genus $g(S)$, for a MED semigroup generated by $a_1 < \ldots < a_k$ can be computed easily, as
$$
g(S) = a_k - \left( \frac{a_2+ \cdots + a_k}{a_1} + \frac{a_1+1}{2} \right).
$$

\ 

Our next aim is to give a characterization for the MED semigroups, in terms of the Apéry sets, in order to be effective. We begin with a definition which puts the focus on the behaviour of single elements, instead of the full semigroup. 

\begin{definition}
Let $S$ be a numerical semigroup, $z \in S$. We will say that $z$ is an Arf element of $S$ if, for all $x,y \in S$ with $x,y \geq z$, $x+y-z \in S$.

When all elements of $S$ are Arf, $S$ will be called an Arf semigroup.
\end{definition}

If no confusion arises, we will simply say \textit{$z$ is Arf} instead of \textit{$z$ is an Arf element of $S$}. Some immediate remarks from the definition are the following:
\begin{enumerate}
    \item[a)] It is equivalent to check the condition for $x,y>z$.
    \item[b)] If $z>f(S)$, $z$ is Arf, as $x+y-z > f(S)$ for all $x,y>z$.
\end{enumerate}

Therefore, in order to see if $S$ is an Arf semigroup, we only have to care about $N(S)$, the set of sporadic (or left) elements of $S$ (actually the non-zero ones), which is in fact a finite set.

\ 

An easy equivalent definition, in terms of Apéry sets, is the following.

\begin{proposition}
Let $S$ be a numerical semigroup, $z \in S  \setminus \{0\}$. Then the following con\-di\-tions are equivalent:
\begin{enumerate}
    \item[a)] $z$ is Arf.
    \item[b)] If $x,y \in \Ap(S,z)$, with $x,y>z$, then $x+y \notin \Ap(S,z)$.
\end{enumerate}
\end{proposition}

\begin{proof}
The direct implication is straightforward from the definitions of Arf elements and Apéry sets. Let us check the reverse one. 

In order to show that $z$, verifying condition b), is Arf, let us take $x,y \in S$ with $x,y>z$. Then, if $x \notin \Ap(S,z)$, $x-z \in S$ and therefore $(x-z)+y \in S$. The same goes if $y \notin \Ap(S,z)$. Therefore, we only have to care about the case $x,y \in \Ap(S,z)$, in which case, condition b) implies $(x+y)-z \in S$. This finishes the proof.
\end{proof}



In more common terms in combinatorial number theory one could say that $z$ is Arf if and only if
$$
\{ x \in \Ap(S,z) \; | \; x>z \big\}
$$
is a sum--free set. Note that exactly the same result holds for $\wAp(S,z)$.

\begin{remark}
For an arbitrary element $s \in S$, being Arf and having $\wAp(S,s)$ sum--free are \textit{not} equivalent. For example $\wAp(S,2m(S))$ is never sum--free, even if $S$ is Arf.
\end{remark}

The following result has a consequence which illustrates nicely the relationship between MED and Arf numerical semigroups. This result can be found in \cite{GR}, within a different context, including more equivalent conditions (which are of no use for our purposes). We include a direct proof for the convenience of the reader.

\begin{proposition}
$S$ is a MED numerical semigroup if and only if $m(S)$ is Arf.
\end{proposition}

\begin{proof}
Assume $S$ is a MED numerical semigroup, minimally generated by $a_1 < \ldots < a_k$ (that is, $m(S) = a_1$). As $e(S) = m(S) = a_1$, the Apéry set $\Ap(S,a_1)$ must be
$$
\Ap(S,a_1) = \left\{ 0,a_2,\ldots,a_k \right\}.
$$

All the non-zero elements are bigger than $a_1$. If we take two of these elements, say $a_i, a_j \in \Ap(S,a_1)$, then $a_i+a_j \notin \Ap(S,a_1)$, for otherwise the set of generators would not be minimal. Therefore, $a_1=m(S)$ is Arf.

Conversely, assume $m(S) = a_1$ is Arf and let its Apéry set be
$$
\Ap(S,a_1) = \left\{ 0=x_1 ,x_2,\ldots,x_{m(S)} \right\}.
$$

This implies, as we saw above, that $\left\{ a_1, x_2,\ldots, x_{m(S)} \right\}$ is a generating set. But being $a_1$ Arf, it is easy to see that $x_2,\ldots,x_{m(S)}$ cannot be expressed (in a non--trivial way) as a sum of elements of $S$. On the one hand, $x_i$ is the minimal element of $S \cap (x_i+\Z/\Z a_1)$, so it cannot be written as 
$$
x_i = y + \lambda a_1, \mbox{ with } y \in S, \; \lambda \in \Z_{> 0}.
$$

On the other hand, from the Arf condition, $x_i$ cannot be written as the sum of two elements of $S$ of non--zero residue in $\Z / \Z a_1$. Therefore $\left\{ a_1, x_2,\ldots, x_{m(S)} \right\}$ is the minimal generating set, and $S$ is a MED numerical semigroup.
\end{proof}

Therefore, all Arf semigroups are MED semigroups.

\begin{example}
Trivially $S_1$ is not a MED numerical semigroup, since $e(S)<m(S)$. Note in this case that
$$
\Ap(S_1, 7) = \big\{ 0, \mathbf{9}, \mathbf{11}, 15, \mathbf{20}, 24, 26 \big\}.
$$

Analogously, 
$$
S_2 = \langle 9, 13, 14, 16, 17, 19, 20, 21, 24 \rangle
$$ 
is not Arf, as
$$
\Ap(S_2, 13) = \big\{ 0, 9, \mathbf{14}, 16, 17, 18, 19, 20, 21, 23, 24, 25, \mathbf{28} \big\},
$$
but
$$
\Ap(S_2, 9) = \big\{ 0, 13, 14, 16, 17, 19, 20, 21, 24 \big\},
$$
so it is MED (we could have also checked that the given generating system is minimal, of course).

If we consider $S_3 = \langle 7, 11, 13, 15, 16, 17, 19 \rangle$, we have
$$
\Ap(S_3, 7) = \big\{ 0, 11, 13, 15, 16, 17, 19 \big\},
$$
so it is MED. Therefore $f(S_3) = 19-7 = 12$ and $N(S_3) = \big\{0,7,11\}$. We have already checked $7$ is Arf, so it only remains to compute
$$
\Ap(S_3, 11) = \big\{ 0, 7, 13, 14, 15, 16, 17, 19, 20, 21, 23 \big\},
$$
where the elements bigger than $11$ are clearly a sum-free subset. Therefore, $S_3$ is Arf.
\end{example}

So, we have translated both categories (MED and Arf semigroups) into similar conditions (some elements verifying the Arf property). While for MED semigroups the Arf property only must hold for $m(S)$, for Arf semigroups all the elements of $N(S)$ must be Arf elements. 

\ 

We finish this section by remembering two important results. The first one can be found in \cite{RGGB}, but we will include a proof, as the reference is apparently hard to access.

\begin{theorem}\label{MED closure}
Let $S$ be a numerical semigroup, $m=m(S)$. There exists a MED semigroup, noted $\MED(S)$, which is the smallest MED semigroup contaning $S$ with multiplicity $m(S)$. It is called the MED closure of $S$.
\end{theorem}

\begin{proof}
Let us call
$$
{\mathcal A}_S = \Big\{ T \; \Big| \; m(T) = m, \; S \subset T, \; T \mbox{ is MED} \Big\}.
$$

Clearly ${\mathcal A}_S$ is finite, because there are only finitely many numerical semigroups containing $S$. Also, ${\mathcal A}_S \neq \varnothing$, because $T = \big\{ 0,m,\rightarrow \big\} \in {\mathcal A}_S$. Here the symbol $\rightarrow$ means that all integers greater than the preceding number are in $T$.

Therefore it makes sense considering
$$
S_0 = \bigcap_{T \in {\mathcal A}_S} T,
$$
which is clearly a numerical semigroup with $m(S_0) = m$. $S_0$ is also a MED semigroup, and this is easily checkable using the fact that $m$ is an Arf element of $T$, for all $T \in {\mathcal A}_S$. 

It is also immediate that $S \subset S_0 \subset T$, for all $T \in {\mathcal A}_S$, by definition of $S_0$. Hence, $S_0$ is the numerical semigroup $\MED(S)$.
\end{proof}

\begin{lemma}\label{lem:union}
    Let $u\in \MED(S)$ and  $S_u$ be the semigroup generetad by $S\cup \{u\}$. Then   $\MED(S_u)=\MED(S)$.
\end{lemma}
\begin{proof} First we note that $m(S_u)=m(S)$ and hence $ {\mathcal A}_{S_u}\subset {\mathcal A}_S$. Moreover  if $u\in\MED(S)$ then $u\in T$ for all $T\in{\mathcal A}_{S}$ so $T\in {\mathcal A}_{S_u}$ if and only if $T\in {\mathcal A}_{S}$.  Hence 
$$
\MED(S)=\bigcap_{T \in {\mathcal A}_S} T=\bigcap_{T \in {\mathcal A}_{S_u}} T=\MED(S_u).
$$
\end{proof}

Once this is done, one can compute $\MED(S)$ in many cases, for a given semigroup $S$, thanks to this result from \cite[Thm.14]{RGGB}:

\begin{theorem}\label{compMED}
    Let $m, r_1, \ldots, r_p \in \N^*$ such that $\gcd \left( m, r_1,\ldots, r_p \right) = 1$. Then
    $$
    \MED \big( \langle m, m+r_1,\ldots, m+r_p \rangle \big) = \big( m + \langle m, r_1,\ldots, r_p \rangle \big) \cup \big\{ 0 \big\}.
    $$
\end{theorem}

While this result is really useful and allows, for instance, to prove that, if $m(S), m(S)+1 \in S$, then $\MED(S) = \big\{ 0,m, \rightarrow \big\}$, it does not necessarily gives a minimal system of generators for $\MED(S)$. The rest of the paper will be devoted to give two different, iterative processes to compute $\MED(S)$ (given by a minimal system of generators) which will hopefully shed also some light on the nature of $\MED(S)$.

\section{Interlude: How to enlarge a numerical semigroup (carefully)}

As our interest will lie on computing explicitely the MED closure, we must have a word on the problem of enlarging a numerical semigroup, in a controlled manner. 

Adding (or removing) elements to a numerical semigroup while preserving the structure is a very interesting problem, with an obvious potential for induction--like strategies. In \cite{RGGJ}, for instance, the problem of adding one element is considered. As for more complicated sets, we have the following. 



\begin{lemma}[\cite{GR}, Ex.1.6]
    Let $S$ be a numerical semigroup. Consider $\PF(S) = \big\{x_1 = f(S) > x_2 > \cdots > x_t \big\}$. Then, for every $r \leq t$, $T = S \cup \{ x_1,\ldots,x_r\}$ is a numerical semigroup.
\end{lemma}

\begin{proof}
   Indeed the set of gaps of $T$ is contained in $\operatorname{G}(S)$, hence is finite.  Finally suppose $a, b\in T$. We have three options:
    \begin{itemize} 
    \item $a,b \in S$. Then $a+b\in S\subset T$.
    \item $a \in S$, $b \notin S$. Then $b \in \PF(S)$ and, either $a=0$ and $a+b=b \in T$ or, by definition  $a+b\in S\subset T$.
    \item $a,b \notin S$. Then, for every $x \in S\setminus \{0\}$, $(a+b)+x = a+(b+x) \in S$, as $a \in \PF(S)$ and $b+x \in S$. Hence $a+b \in \PF(S)$ and, as $a+b>a,b$, it must hold $a+b \in T$.
    \end{itemize}
The three cases above prove the result.
\end{proof}

Not only we have a bigger semigroup from $S$, we can also have some control on its main invariants. Let us name, following the above notation
$$
\PF(S)_r = \big\{ x_1,\ldots,x_r\},
$$
where $\PF(S) = \big\{x_1 = f(S) > x_2 > \cdots > x_t \big\}$, and $S_r = S \cup \PF(S)_r$. Then, by definition,
$$
g(S_r) = g(S)-r,
$$
but also have the following result.

\begin{theorem}
With the notation as above we have 
$$
c(S) \leq 2g(S)-r+1.
$$
\end{theorem}

\begin{proof}
Let us consider the set $I=S\setminus\{0\}$ which is clearly an ideal of $S$. Now, it is also obvious that $I$ is an ideal of $S_r$. 

\ 

Indeed, if $a \in I$ and $b\in S_r$ then either $b\in S$ and then $a+b\in I$ because $I$ is an ideal of $S$, or $b\in \PF(S)_r$, and so $a+b\in S$ and $a+b\ne 0$ so $a+b\in I$. We can now apply the following result \cite[Thm.3]{BLV}:

\begin{theorem}
    Let $H$ be a numerical semigroup and $J \subset H$ an ideal of $H$. Then $c(J) \leq 2g(H) + |H\setminus J|$.
\end{theorem} 

We will apply the lemma for $J=I$ and $H=S_r$. Now, $c(I)=c(S)$, while $g(S_r)= g(S)-r$ and $|S_r \setminus J|=r+1$ and so 
$$
c(S) = c(I) \leq 2g(S_r) + |S_r \setminus J| = 2 (g(S)-r) + r + 1 = 2 g(S) - r + 1,
$$
which proves the result.
\end{proof}

\begin{corollary}
    Given a numerical semigroup $S$, $f(S) \leq 2g(S)-t(S)$, and 
    \begin{equation}\label{eq:ideal}
    2n(S) + t(S) - 1 \leq c(S) \leq 2g(S)-t(S)+1.
    \end{equation}
\end{corollary}

\begin{remark} 
    For the so--called symmetric semigroups we have in fact an equality in (\ref{eq:ideal}) as they verify \cite[Ch.3.1]{GR}
    $$
    g(S) = \frac{c(S)}{2}, \quad t(S)=1,
    $$
    and in fact any of these two conditions can be used to define the family of symmetric semigroups.
    
    However, the equality in (\ref{eq:ideal}) is not a characterization of symmetric semigroups, as shown by
    $$
    S_5 = \langle 5,6,7,9 \rangle = \big\{ 0,5,6,7,9,\rightarrow \big\},
    $$
    with $f(S)=8$, $g(S)=5$, $\PF(S_5)= \big\{ 4,8 \big\}$. 
    
    In \cite[Thm.7]{BLV} we can find a (multiple) characterization of the equality case. Translated to our situation, we have an equality in (\ref{eq:ideal}) if and only if
    $$
    \PF(S) = \Big\{ a \in S_{t(S)} \; | \; 2g - t(S) - 1 - a \in S_{t(S)} \Big\}.
    $$
\end{remark}

In general, of course, if one adds elements outside $\PF(S)$ (or if does not do in the orderly way depicted in the above lemma), the semigroup structure might well be lost. This will be illustrated in our first MED closure process.

\section{Computing $\MED(S)$: The Apéry saturation}

We will use the Apéry sets for a general construction on this regard.

\begin{definition}
Let $S$ be a numerical semigroup, $s \in S$, with 
$$
\wAp(S,s) = \big\{ w_1 , \ldots , w_s \big\}.
$$

We will call the Apéry saturation of $S$ (w.r.t. $s$) the set
$$
S \cup \big\{ w_k-s \; \big| \; \mbox{exists } (w_i,w_j,w_k) \mbox{ with } w_i+w_j = w_k \big\}.
$$
\end{definition}

We will give now an algorithm for computing $\MED(S)$ for a given numerical semigroup $S$, with $m(S) = m$, based on the results from the previous sections. For a given numerical semigroup $S$, with $m(S)=m$, we consider the following operations:

\begin{itemize}
    \item[Step 1)] Compute $\wAp(S,m) = \big\{ m=a_0,\ldots,a_{m-1} \big\}$, with $a_i \equiv i \mod m$.
    
    \item[Step 2)] Consider 
    $$
    \widehat{S} = \Big\langle S \cup \big\{ a_k-m \; \big| \; \mbox{exists } (a_i,a_j,a_k) \mbox{ with } a_i+a_j = a_k \big\} \Big\rangle,
    $$
\end{itemize}

And we define inductively the following sequence:
$$
S^1 = S, \quad S^{n+1} = \widehat{S^n}, \;\; \forall n \geq 1.
$$

\

First, let us mention some basic facts of this procedure:
\begin{itemize}
    \item For every $n \geq 1$, $m(S^n) = m$, as in the Apéry saturation we are adding elements
    $$
    a_k - m = (a_i+a_j) - m, \mbox{ with } a_i,a_j \geq m.
    $$
    \item For every $n \geq 1$, we have that $S^n \subseteq S^{n+1}$. In fact, $S^{n+1} = S^n$ if and only if $S^n$ is a MED numerical semigroup, as it implies that $\wAp(S^n,m)$ is a sum--free set.
    \item If we order $\wAp(S,m)$ with the partial ordering $\leq_S$ introduced at the end of Section 1, it is immediate to see that we are adding to the new semigroup $\widehat{S}$ one element for every element in $\wAp(S,m)$ which is not minimal. In particular, we are adding $\PF(S)$, except those of the form $a_i+m(S)$, where $a_i$ is in the generating set.
    \item The Apéry saturation is \textit{not}, in general, a numerical semigroup (see Example \ref{notNS} below)
\end{itemize}

We can show now the fundamental result of this section: that the previous algorithm computes, in fact, $\MED(S)$.

\begin{theorem}
The previous sequence becomes stationary
$$
S = S^1 \subsetneq S^2 \subsetneq \ldots \subsetneq S^{l-1} \subsetneq S^l = S^{l+1} = \ldots,
$$
with $S^l = \MED(S)$.
\end{theorem}

\begin{proof}
The sequence is increasing with respect to set--inclusion and the set of numerical semigroups containing $S$ with multiplicity $m$ is finite, so we must get to a stationary term, therefore an element of ${\mathcal A}_S$ (in the notation of Theorem \ref{MED closure}). We must prove that this stationary term is, in fact, the MED closure of $S$.

Assume then that $T \in {\mathcal A}_S$ and suppose we have
$$
\wAp(S,m) = \big\{ m=a_0,\ldots,a_{m-1} \big\}
$$
with $a_i+a_j = a_k$ for some $\{i,j,k\} \subset \big\{0,\ldots,m-1\}$. Then $a_i,a_j \in T$ and $T$ is a MED numerical semigroup. Therefore $m$ is an Arf element of $T$ and 
$$
a_i + a_j - m = a_k - m \in T.
$$

This shows that if $S \subset T$, then $\widehat{S} \subset T$ as well and therefore ${\mathcal A}_S = {\mathcal A}_{\widehat{S}}$. Inductively,
$$
{\mathcal A}_S = {\mathcal A}_{S^1} = \ldots = {\mathcal A}_{S^l}
$$
and therefore $\MED(S) = \MED(S^1) = \ldots = \MED(S^l) = S^l$. This finishes the proof.
\end{proof}

Let us give then some examples to illustrate how the procedure works. 

\begin{example}
Take $S_4 = \langle 7, 24, 33 \rangle$, where
$$
\wAp(S_4,7) = \big\{ 7,57,72,24,81,33,48 \big\}
$$

We have the following sums in $\wAp(S_4,7)$:
$$
48 = 24+24, \quad 57 = 24+33, \quad 72 = 48+24, \quad 81 = 48+33.
$$

Therefore we would have to add $41$, $50$, $65$ and $74$ to the generators of $S_4 = S_4^1$ in order to get $S_4^2$. In particular we are adding $\PF(S_4) = \big\{ 65, 74 \big\}$, as expected (as neither 72, nor 81 are generators). Note, however, that $74$ is redundant and hence
$$
S_4^2 = \langle 7, 24, 33, 41, 50, 65 \rangle.
$$

Now,
$$
\wAp(S_4^2,7) = \big\{ 7,50,65,24,74,33,41 \big\},
$$
where we only have one sum, precisely $74 = 33+41 = 50+24$. That situation of having more than one sum giving the same element can easily happen, but we are actually only interested in $74$ for the next step. Therefore,
$$
S_4^3 = \langle 7, 24, 33, 41, 50, 65, 67 \rangle,
$$
which is not MED, as 
$$
\wAp(S_4^3,7) = \big\{ 7,50,\mathbf{65},\mathbf{24},67,33,\mathbf{41} \big\}.
$$

Iterating one more time we finally get
$$
S_4^4 = \langle 7, 24, 33, 41, 50, 58, 67 \rangle,
$$
which is MED, as $\wAp(S_4^3,7)$ is precisely the set of minimal generators.
\end{example}

\begin{example}
The preceding example showed a sequence with a strictly increasing dimension (until it became stationary, of course). This might not be the case. Take, for instance:
$$
S_5 = \langle 4,7,17 \rangle,
$$
where
$$
\wAp(S_5,4) = \big\{ 4,17,14,7 \big\},
$$
and we have $14=7+7$ and hence we have to add $10$ to make the Apéry saturation, and then
$$
S_5^2 = \langle 4,7,10 \rangle.
$$
\end{example}

\begin{example}\label{notNS}
It can happen that some elements in $\PF(S)$ are not added to $\widehat{S}$. Take, for instance \cite[Ex.1.27]{GR}, 
$$
S_6 = \langle 11, 12, 13, 32, 53 \rangle,
$$
where $\PF(S_6) = \big\{ 21, 40, 41, 42 \big\}$ and
$$
\wAp(S_6,11) = \big\{ 11, 12, 13, 25, 26, 38, 39, 51, 52, 53, 32 \big\}.
$$

The Apéry saturation of $S_6$ w.r.t. 11 is given by
$$
S_6 \cup \big\{ 14, 15, 27, 28, 40, 41 \big\},
$$
so we have failed to add, not only all the set $\PF(S_6)$, but $f(S_6) = 42$. However, once we consider the semigroup $S_6^2$, we have $42 \in S_6^2$, but $\PF(S_6) \not\subset S_6^2$ as $21 \notin S_6^2$.
\end{example}
\begin{example}
    Consider again $S_1 = \langle 7, 9, 11, 15 \rangle$, and remember
    $$
    \wAp(S_1, 7) = \big\{ 7, 15, 9, 24, 11, 26, 20 \big\},
    $$
    therefore the Apéry saturation is
    $$
    S_1 \cup \big\{ 13, 17, 19 \big\}
    $$
    and we get to 
    $$
    \MED(S_1) = \langle 7, 9, 11, 13, 15, 17, 19 \rangle
    $$
\end{example}

\section{Computing $\MED(S)$: An effective solution}

After giving an Apéry-based solution to the problem we aim to give an effective one, in the sense that we can control (or, at the very least, bound) the number of operations involved.

\begin{remark}
    Let $S$ be a numerical semigroup, and suppose $u,v\in G(S)$ are such that $u>v$ and  $u \equiv v \pmod {m(S)}$. Then, as $u=v+km(S)$ for some $k\geq 0$, we have $\langle S \cup \{u\} \rangle \subseteq \langle S \cup \{v\} \rangle$.
\end{remark}

\begin{lemma}\label{lem:conmed} 
Let $S$ be a numerical semigroup, and $\{a_1,\dots,a_n\}\subset S$. Then for every set of non negative integers $\{k_1,\dots, k_n\}$  with 
$$
K = \sum_{i=1}^n k_i \ge 2,
$$ 
we have
$$
\sum_{i=1}^n k_i a_i - (K-1) m(S)\in \MED(S).
$$
\end{lemma}

\begin{proof} 
We will prove it by induction in $K$. If $K=2$ the result follows since $m(S)$ is Arf in $\MED(S)$. 

Suppose it is true for a certain $K\ge 2$ and we want to prove it for the next case $K+1 = \sum_{i=1}^n k_i$. We can assume, without loss of generality, that  $k_1\ge1$. Then
\begin{eqnarray*}
\sum_{i=1}^nk_ia_i-Km(S)&=&a_1-m(S)+\sum_{i=1}^nk_i'a_i-(K-1)m(S)\\
&=&A_1+A_2-m(S),
\end{eqnarray*}
where $k_i'=k_i$ if $i\ge 2$, $k_1'=k_1-1$, 
$$
A_2=\sum_{i=1}^nk_i'a_i-(K-1)m(S)\in \MED(S),
$$ 
by the induction hypothesis and  $A_1=a_1\in S\subset \MED(S)$. Hence 
$$
\sum_{i=1}^nk_ia_i-Km(S)\in \MED(\MED(S))=\MED(S)
$$
as we wanted.
\end{proof}

\



\begin{lemma}\label{lem:ds} 
Let  $a_1 < \cdots < a_{n}$ be a minimal set of generators of a semigroup $S$ and suppose $a$ is in the minimal set of generators of $\MED(S)$. Then we can write
$$
a=\sum_{i=2}^{n} k_i a_i-(K-1)m(S),
$$ 
for some $k_i \geq 0$, and 
$$
K = \sum_{i=2}^n k_i \le \frac{c(S)}{a_2-m(S)}.
$$
\end{lemma}

\begin{proof}

By construction for $2\le i\le n$, $a_i=a_1+r_i$, where $r_i\ge 1$ and such that $\gcd (a_1,r_2,\dots,r_n)=1$. By Theorem \ref{compMED} 
$$
a= (d+1) a_1 + \sum_{i=2}^nk_ir_i,
$$ 
for some $d, k_i\ge 0$ for $i=2,\dots, n$. Letting $K=\sum_{i=2}^nk_i$ we see that 
$$
a=dm(S)+\sum_{i=2}^nk_i(r_i+m(S))-(K-1)m(S).
$$
But by the previous lemma, 
$$
\sum_{i=2}^nk_i(r_i+m(S))-(K-1)m(S)\in \MED(S),
$$
and hence $d=0$ by minimality of $a$ and the first part of the lemma follows. 

It remains to prove the upper bound. Now since $a$ is a minimal generator of $\MED(S)$ and must have been added in some Apéry saturation, it is direct that $a \le c(S)+m(S)$. Writing as before
$$
a=\sum_{i=2}^{n} k_i a_i-(K-1)m(S),
$$ 
we have then
$$
Ka_2-(K-1)m(S) \le \sum_{i=2}^{n} k_i a_i-(K-1)m(S) =a \le c(S)+m(S),
$$
and the result now follows.
\end{proof}

In order to estimate complexity orders later on an upper bound for the conductor of the semigroup might be handy and there are many available such bounds in the literature depending on the usual invariants of $S$. For example we know by \cite{Selmer} that, if $S$ is generated by $m(S)=a_1 < \cdots < a_n$, then 
\begin{equation}\label{eq:cs}
c(S) \le 2a_n\left[\frac{m(S)}{n}\right]-m(S)+1.
\end{equation}

From this bound, and the previous result, let us write
$$
d(S) = \frac{\displaystyle 2a_n\left[\frac{m(S)}{n}\right]-m(S)+1}{a_2-m(S)},
$$
although the numerator can be modified to include any other bound (or expression) of $c(S)$, should this be convenient. In particular, note that $K \leq d(S)$ in the statement of Lemma \ref{lem:ds}.

\ 

We can now present an alternative algorithm to compute $\MED(S)$, one whose number of steps can be controlled by the previous results. The input will be $S$, given by a minimal set of generators $a_1 < \cdots <a_n$. This system contains $n-1$ non--zero classes modulo $a_1 = m(S)$.

\begin{lemma} \label{lem:representante}
Let $S=\{m=a_1<\dots <a_n\}$ given by a minimal set of generators and $M=\{i_1,\dots,i_{n-1}\}$ the non--zero residue classes modulo $m$ represented by the generators. 

\medskip

For each $1 \le i < m$, let us define $c_i$ to be any element of $S$ with $c_i\ge d(S)\sum_{i=2}^na_i$ and such that $c_i \equiv i \pmod m$. 
Consider now
$$
K_i=\max \left\{ \sum_{j=1}^nk_{ij} \, : \, \sum_{j=1}^nk_{ij}a_j=c_i \right\}
$$
and $g_i=c_i-(K_i-1)m$. Then
$$
\MED(S)=\{m=g_0,\dots, g_{m-1}\}.
$$
\end{lemma}

\begin{proof}
First note that $g_i$ is independent of the selection of $c_i$. Indeed, by Lemma  \ref{lem:ds} we know that the minimal element of MED(S) in the congruent class $i$ is of the 
$g_i=C_i-(K-1)m$ for some 
$$
C_i=\sum_{j=2}^nk_ia_i\le d(S)\sum_{j=2}^na_i.
$$

Hence for any $c_i$ in the conditions of the lemma, we will have $c_i=C_i+tm$ for some $t\ge 0$ but then, writing $k_1=t$, 
$$
c_i=\sum_{j=1}^nk_ia_i, \quad \sum_{j=1}^nk_i=K+t
$$ 
and, therefore,
$$
c_i-(t+K-1)m=C_i-(K-1)m=g_i.
$$

Now, note that by Lemma \ref{lem:conmed}, $g_i\in \MED(S)$, and the semigroup generated by $\{m=g_0,\dots, g_{m-1}\}$ contains $S$, so the proof would end if we proof that  $\{g_0,\dots, g_{m-1}\}$ is a sumfree set. But 
$$
g_i+g_l=c_i+c_l-(K_i+K_l-2)m>c_i+c_l-(K_i+K_l-1)m\ge c_i+c_l-(K_{i+l}-1)m=g_{i+l}.
$$ 

For the last inequality note that we can take $c_{i+l}=c_i+c_l\ge d(S)\sum_{i=2}^na_i$, hence 
$$
c_{i+l}=\sum_{j=1}^n \left( k_{ij}+k_{lj} \right)a_j
$$ 
and then by maximality $\sum_{j=1}^n(k_{ij}+k_{lj})\le K_{i+l}$, so the sum is not in $\{g_0,\dots, g_{m-1}\}$.
\end{proof}




\begin{theorem} 
Let $m$ be a positive integer fixed and $S=\{m=a_1<\dots<a_n\}$ be a semigroup given by a minimal set of generators. Then we can compute $\MED(S)$ in time linear in the size of the generators. 
\end{theorem}

\begin{proof} 
The first step in the algorithm will be selecting for each congruence class $1\le i<m-1$ an element $c_i\in S$, $c_i\ge d(S)\sum_{i=2}^na_i$. Then we need to find the maximum $\sum_{i=1}^nk_{ij}$ among the possible representations 
$$
c_i=\sum_{i=1}^nk_{ij}a_j.
$$

It is clear that if $k_{ij}=r_{ij}+u_{ij}a_1$ with $0\le r_{ij}<m-1$, then
$$
k_{i1}a_1+k_{ij}a_j=(k_{i1}+u_{ij}a_j)a_1+r_{ij}a_j
$$
and 
\begin{equation}\label{eq:maxk}
k_{i1}+k_{ij} = k_{i1}+r_{ij}+u_{ij}a_1 \le k_{i1}+u_{ij}a_j+r_{ij}
\end{equation}

In particular the maximum will be attained for some values $k_{ij}\le m-1$ for any $2 \le j \le n$. Hence, we will select at random $k_{ij}$ for $2 \le j \le n-1$ and 
$$
k_{in} \equiv i-\sum_{j=2}^{n-1}k_{ij}a_j.
$$
Finally $k_{i1}=(\sum_{j=2}^{n}k_{ij}a_j-i)/a_1$

\

Hence, since $k_{ij} \le m-1$, 
$$
\sum_{j=2}^{n-1}k_{ij}\le (n-2)(m-1),
$$
so it is like selecting at most $(n-2)(m-1)$ elements among $a_2,\dots, a_{n-1}$ with repetitions which is
$$
\binom{(n-2)(m-1)}{n}<\frac{2^{n\log m}}{\sqrt n}
$$ 
for the last inequality see \cite[p.309]{WS} and we need to do it $m-2$ times, one for each congruence class. Hence, since dividing $(\sum_{j=2}^{n}k_{ij}a_j-i)$ by $a_1$ takes $O(\log m\log a_n)$ bit operations, finding the maximum will be bouded by 
$$
O \left( \log m\log a_n+\frac{m}{\sqrt n}2^{n\log m} \right)
$$ 
bit operations. 
\end{proof}

\begin{remark} 
With the bound given, the dependence in $m$ is exponential. However, we believe this bound can be improved, and we leave this task for future work. 
\end{remark}


\begin{example}
We will use $S_4 = \langle 7,24,33 \rangle$. Then $d(S)<8$ and $\sum_{i=2}^na_i=57$. Take $c_i$ from the following table
\begin{center}
\begin{tabular}{c c c c c c c }
$i$  &$1$&$2$&$3$&$4$&$5$&$6$
\\
$c_i$ &$456$ &$576$ & $528$ &$648$& $495$& $552$
\end{tabular}
\end{center}
Now:
\begin{itemize}
    \item $456=33+24+57\times 7$, so $g_1=456-58\times 7=50$,
    \item $576=3\times 24+7\times 72$, so $g_2=576-74\times 7=58$,
    \item $528=72\times 7+24$ so $g_3=528-72\times 4=24$,
    \item $648=2\times 24+33+7\times 81$, so $g_4=648-83\times 7=67$,
    \item $495=33+66\times 7$, so $g_5=495-66\times 7=33$, and
    \item $552=2\times 24+72\times 7$, so $g_6=552-73\times 7=41$.
\end{itemize} 

Hence $\MED(S)=\langle 7,24,33,50,58,67,41 \rangle$.
\end{example}

\


\begin{example}
Let us take now $S_6 = \{11,12,13,32,53\}$.
\end{example}
Then $d(S)=208$ y $\sum_{i=2}^na_i=110$ and since $22880=208\times 110=2080\times 11$  we can  take
$$
c_i=(22880+i)12=12\times i+(2080\times 12)11
$$
Then $g_i=c_i-(2080\times 12+i-1)11=11+i$, so
$$
\MED(S)=\langle 11,12,13,14,15,16,17,18,19 \rangle.
$$

\








\begin{example}
$S=\{4,7,17\}$ 
\end{example}
In this case $d(S)\le 11$ and $a_2+a_3=24$ so take $K=264=4\times 66$. Take $c_1=289$, $c_2=306$, $c_3=323$. Then:
\begin{itemize}
    \item $c_1=3\times 7+67\times 4$, so $g_1=13$, 
    \item $c_2=2\times 7+73\times 4$, so $g_2=10$, and finally 
    \item $c_3=7+79\times 4$ so $g_3=7$,
\end{itemize}
therefore $\MED(S)=\{4,7,10,13\}$.

\section{Final remarks}

In this paper we have isolated the Arf property for elements of a numerical semigroup, a property can be expressed in terms of the Apéry set of the element. This allows us to illustrate clearly the difference between MED and Arf numerical semigroups. 

\ 

As an application of this, we have described two new algorithms to compute the MED closure of a given semigroup, taking advantage of the previous description and highlighting some new features of the minimal system of generators of said clousure. For the last of them we can actually compute its complexity in terms of the generators. To our knowledge, it is the first result in the literature in this regard.

\

\noindent \textbf{Conflict of interest and data availability statement:} Not applicable.

\end{document}